\documentclass[11pt,a4paper,twoside]{amsart}
\usepackage[english]{babel}
\usepackage{latexsym,layout,verbatim,amssymb,amsmath,setspace,cite}
\usepackage{amsthm,mathrsfs}
\usepackage{cleveref,enumerate,fontenc}
\usepackage[utf8]{inputenc}
\usepackage[T1]{fontenc}
\newtheorem{theorem}{Theorem}[section]

\newtheorem{corollary}[theorem]{Corollary}

\newtheorem{definition}[theorem]{Definition}
\newtheorem{lemma}[theorem]{Lemma}
\newtheorem{proposition}[theorem]{Proposition}
\newtheorem{remark}[theorem]{Remark}

\newtheorem{example}[theorem]{Example}
\newcommand{\Pd}{\mbox{$\mathcal{P}_{\Delta}$}~}
\def\cal#1{\mathcal{#1}}

\def\erre{\mathbb{R}}
\def\vp{\varepsilon}

\begin{document}

\title[An extension of the Birkhoff  integrability ...]{An extension of the Birkhoff  integrability\\ for multifunctions}

\author[Candeloro]{\bf Domenico Candeloro}
\address{Department of Mathematics and Computer Sciences \\ University of Perugia\\
Via Vanvitelli, 1 - 06123 Perugia (Italy)}
\email{candelor@dmi.unipg.it}

\author[Croitoru]{\bf Anca Croitoru}
\address{\rm Faculty of Mathematics,\\ ''Al.I. Cuza'' University\\ Bd. Carol I, no 11, Ia\c{s}i, 700506, 
(Rom\^{a}nia)}
\email{croitoru@uaic.ro}

\author[Gavrilu\c{t}]{\bf Alina Gavrilu\c{t}}
\address{\rm Faculty of Mathematics,\\ ''Al.I. Cuza'' University \\ Bd. Carol I, no 11, Ia\c{s}i, 700506, 
(Rom\^{a}nia)}
\email{gavrilut@uaic.ro}

\author[Sambucini]{\bf Anna  Rita Sambucini}
\address{\rm Department of Mathematics and Computer Sciences \\ University of Perugia\\
Via Vanvitelli, 1 - 06123 Perugia (Italy)}
\email{anna.sambucini@unipg.it}


\subjclass[2010]{28B20, 58C05}
\keywords{Birkhoff simple integral, Gould integral, Mc Shane integral,countably additive measures, finitely additive measures, atoms, pointwise non atomic measures.}
\date{July 23, 2015}

\begin{abstract} 
A comparison between  a set-valued Gould type and simple Birkhoff integrals of $bf(X)$-valued multifunctions
with respect to a non-negative set functionis given. Relationships among them and 
 Mc Shane multivalued integrability is given under suitable assumptions.
\end{abstract}
\maketitle

\section{Introduction}\label{one}
After Aumann, set-valued integrals have been defined and studied by
many authors using different techniques: see for example 
\cite{cr2004,cr2005,tre,cg2014,cgsub1,g2008,gcsub,vm,pgcsub,c2013,ccgs2014} for the Birkhoff and Gould multivalued integrability, 
\cite{DPM,dpmb1,dpm3,dpm4,bs2004,bs2012,bcs2014,cs2014,cs2015,ebs2004,ebs2005}
for the Kurzweil-Henstock, Mc Shane, Pettis integrability and their relations.\\
These notions have shown to be useful when modelling some theories in different fields as  optimal
control and mathematical economics, see for example \cite{am,dps1,dpms2014,cc}.
Moreover the theory of set-valued functions has interesting and important applications also
in many theoretical or practical domains such as 
statistics, fixed point theory, stochastic processes, information sciences and optimization.
Integrals of multifunctions can be used as an aggregation tool
when dealing with a large amount of information fusing and with data mining
problems such as programming and classification.
In decision making theory, if the available information about the states
of nature is given by means of a monotone measure, the integral of the payoff
function (or multifunction) gives an average value that could be used to find
the best solution for multi-criteria decision making.
\\
A comparison between these integrals in the vector valued  case was given for example in \cite{DPMILL,fremlin,DPVM,pot2007,pot2006,r2005,r2009a,r2009b,cs2014,cs2015}.
Parallel to this study a comparison in the multivalued case was obtained in \cite{ms2001,ms2002,bs2012,dpm4}.\\
In this paper this field of research is continued and a new definition of the Birkhoff multivalued integrability is given, when the hyperspace is that of bounded subsets of a Banach space $X$ ($b(X)$).  A similar definition was introduced also in \cite{cgsub1}:
the definition adopted here is simpler and (at least formally) weaker: later (Definition \ref{ex3.5})  the differences between them will be described more precisely.
 Moreover, a comparison of this integral, called Birkhoff simple, is given, with Gould and Mc Shane integrability.\\
After some preliminaries in sections \ref{two} and \ref{three} regarding respectively set functions and the properties of Gould and Birkhoff simple multivalued integral,
in subsections \ref{three.1} and \ref{three.2} a comparison between Gould and Birkhoff simple-integrability is stated both in countably additive case and under the weaker condition of
 countable null null additivity.
While in section \ref{four} a comparison is obtained with the Mc Shane multivalued integral when the set function $\mu$ is pointwise non atomic, countably additive and the multifunctions $F$ take value in $b(X)$.

\section{Preliminaries}\label{two}

Let $T$ be an abstract, nonvoid set, $\mathcal{A}$ an algebra of subsets of $
T$ and \mbox{$\mu :\mathcal{A}\rightarrow \mathbb{R}_0^+$} a non-negative set
function with $\mu (\emptyset )=0.$ 
In this paper the set functions $\mu$ considered are mainly finitely additive or countably additive, for  all unexplained definitions on set functions see for example \cite{cg2014}.\\

\begin{definition}\label{ex2.2} \rm  \cite{ggg}
 Let $\mu :\mathcal{A}\rightarrow  \mathbb{R}_0^+$ be a non-negative set function, with $\mu (\emptyset )=0.$
\textit{The variation} $\overline{\mu}$ of $\mu $ is the set function 
$\overline{\mu }:\mathcal{P}(T)\rightarrow  \mathbb{R}_0^+$ defined by 
$\overline{\mu }(E)=\sup \{\sum\limits_{i=1}^{n}\mu (A_{i})\}$, 
where the supremum is extended over all families of
pairwise disjoint sets $\{A_{i}\}_{i=1,\ldots,n}\subset \mathcal{A}$,
with $A_{i}\subseteq E$, for every $i\in \{1,2,\ldots,n\}.$
$\mu $ is said to be \textit{of finite variation on }$\mathcal{A}$ if $\overline{\mu }(T)<\infty $.
\end{definition}
For the properties of  $\overline{\mu}$ see for example \cite[Remark 2.3]{cg2014}

\begin{definition}\label{ex2.4} \rm 
 Let $\mu :\mathcal{A}\rightarrow \lbrack 0,\infty )$
be a non-negative set function.
A set $A\in $ $\mathcal{A}$ is said to be an \textit{atom} of $\mu $ if 
$\mu (A)>0$ and for every $B\in $ $\mathcal{A}$, with $B\subset A$, it is
$\mu (B)=0$ or $\mu (A\backslash B)=0$. \\
$\mu $ is said to be \textit{finitely purely atomic} if there is a
finite family $(A_{i})_{i=1,\ldots,n}$ of pairwise disjoint atoms of 
$\mu $ so that $T=$ $\cup_{i=1}^n A_{i}$;\\ 
-\textit{\ null-additive} if $\mu (A\cup B)=\mu (A)$, for every $A,B\in 
\mathcal{A}$, with $\mu (B)=0$;\\ 
-\textit{null-null-additive } if $\mu (A\cup B)=0$, for every $A,B\in 
\mathcal{A}$, with $\mu (A)=\mu (B)=0$;\\
-$\sigma$-\textit {null-null-additive } if $\mu (\cup_{n=1}^{\infty} A_{n})=0$, for every $A_{n}\in \mathcal{A}$, $n\in 
\mathbb{N}$, with $\mu (A_{n})=0$;\\
-\textit{continuous from below} if $\lim_n\mu(A_n)=\mu(\bigcup_n A_n)$ whenever $(A_n)_n$ is an increasing sequence in $\cal{A}$.
\end{definition}

 \textit{A finite (countable}, respectively) 
\textit{partition of} $T$ is a finite (countable, respectively) family of
nonempty sets $P=\{A_{i}\}_{i=1, \ldots, n}$ ($\{A_{n}\}_{n\in \mathbb{N}}$
, respectively) $\subset \mathcal{A}$ such that $A_{i}\cap A_{j}=\emptyset
,i\neq j$ and $\cup_{i=1}^{n} A_{i}=T$ ($\cup_{n\in 
\mathbb{N}}A_{n}=T$, respectively). \\
 If $P$ and $P^{\prime }$ are two finite (or countable) partitions of $T$
, then $P^{\prime }$ is said to be \textit{finer than} $P$, denoted\ by $
P\leq P^{\prime }$ (or $P^{\prime }\geq P)$, if every set of $P^{\prime }$
is included in some set of $P$.\\
  The \textit{common refinement} of two finite or countable partitions $
P=\{A_{i}\}$ and $P^{\prime }=\{B_{j}\}$ is the partition $P\wedge P^{\prime
}=\{A_{i}\cap B_{j}: A_{i}\cap B_{j}\neq\emptyset\}$.

Obviously, $P\wedge P^{\prime }\geq P$ and $P\wedge P^{\prime }\geq
P^{\prime }.$ We denote by $\mathcal{P}$ the class of all finite partitions
of $T$.

\begin{proposition}\label{ex2.6}
{\rm  \cite[Proposition 2.7, Corollary 2.8]{cg2014}} Suppose $\mu :\mathcal{A}\rightarrow \mathbb{R}_0^+$ 
is null-additive and monotone and $A\in \mathcal{A}$ is an atom
of $\mu .$
\begin{enumerate}[\rm  \ref{ex2.6}.i)]
\item If $\{B_{i}\}_{i=1}^{n}$ is a partition of $A$,
then there exists only one $i_{0} \in \{ 1, \ldots, n \}$ so that $\mu (B_{i_{0}})=\mu (A)$ and 
$\mu (B_{i})=0$, for every $i=\{ 1, \ldots, n\}$, with $i\neq i_{0}.$
\item Suppose $\mu $ is $\sigma$-null-null-additive. If $\{B_{n}\}_{n\in 
\mathbb{N}}$ is a partition of $A$,
then there exists only
one $n_{0}\in \mathbb{N}$ so that $\mu (B_{n_{0}})=\mu (A),B_{n_{0}}$ is an
atom of $\mu $ and $\mu (B_{n})=0$, for every $n\in \mathbb{N}\backslash
\{n_{0}\}.$
\item If $A\in \mathcal{A}$ is an atom
of $\mu $, then $\overline{\mu }(A)=\mu (A). $
\end{enumerate}
\end{proposition}
The next definitions will be useful later.
The Gould integral will be introduced first.
\begin{definition}\label{ex3.1} \rm 
 \cite[Definition 3.6]{pgcsub} Suppose $F:T\rightarrow b(X)$ is
a multifunction and $\mu:\mathcal{A}\to \mathbb{R}^+_0$ is any
 set function defined on the algebra $\mathcal{A}$.
\begin{enumerate}[\rm \bf \ref{ex3.1}.i)]
\item  $F$ is said to be \textit{Gould integrable (on} $T)$ \textit{relative to }$\mu $ if there exists a set $I\in bf(X)$ such that for every 
$\varepsilon >0$, there is a finite partition $P_{\varepsilon }$ of 
$T$ such that all the sets of $P_{\varepsilon }$ have finite measure and, moreover,
 for every partition of $T$, $P:=\{A_{i}\}_{i=1}^{n}$, with $P\geq P_{\varepsilon }$ and
every $t_{i}\in A_{i},i\in \{1,2,\ldots,n\}$, it holds
$$h(\sum_{i=1}^n F(t_{i})\mu (A_{i}),I)<\varepsilon.$$
The set $I$ is called \textit{the Gould integral of} $F$ \textit{\ on } $T$
\textit{\ with respect to} $\mu $, denoted by $(G)\int_{T}Fd\mu ;$

\item If $B\in \mathcal{A},F$ is said to be \textit{Gould integrable on} $B$
if the restriction $F_{|B}$ of $F$ to $B$ is Gould integrable on $(B,\mathcal{A}_{B},\mu _{B}).\bigskip $
\end{enumerate}
\end{definition}
Next, the notion of {\em integrable set function} is given.

\begin{definition}\label{muintegrabile}\rm
The set function $\mu$ is said to be {\em integrable} if the (real) constant function $1$ is Gould-integrable (see \cite[Definition 4.5]{cg2014}) with respect to $\mu$.

In other words,  $\mu$ is integrable if and only if there exists in $\erre^+_0$ the limit $\lim_{>}s(\mu,P)$, where $P$ runs along all finite partitions of $T$, $P=\{A_1,...,A_k\}$, ordered by refinement, and $s(\mu,P):=\sum_{i=1}^k\mu(A_i)$.
  
If $\mu$ is integrable, then every {\em restriction} $\mu_A$ is, for $A\in \cal{A}$, where $\mu_A(B):=\mu(A\cap B)$ for every $B\in \cal{A}$, and the integral is an additive set function on $\cal{A}$. This function will be denoted by $\int_A\mu$ or also $\Psi(A)$.
\end{definition}

\begin{remark}\label{fatti-p}\rm
The following facts are easy to prove, for a non-negative, integrable, increasing set function $\mu$ (see also \cite[Definition 4.5]{ccgs2014} for a similar concept):
\begin{description}
\item[\rm \ref{fatti-p}.1)]
Fix $\vp>0$ and let $P_{\vp}\equiv\{E_1,...,E_n\}$ be any partition of $T$ such that
$$|\Psi(T) -\sum_{i=1}^K\mu(C_i)|\leq \vp$$
for every partition $\{C_1,...,C_K\}$ finer than $P_{\vp}$. Such a partition $P_{\vp}$ will be called an $\vp$-approximated partition for $\mu$ in $T$. It turns out that, for every set $A\in \cal{A}$, the partition 
$P_{\vp,A}:=\{E_1\cap A,E_2\cap A,...,E_n\cap A\}$ is a
$2\vp$-approximated partition for $\mu_A$.
\item[\rm \ref{fatti-p}.2)] \
A kind of {\em Henstock Lemma} holds, for $\Psi$: for every $\vp>0$ there exists a partition $P_{\vp}$ of $T$, such that, for every finer partition $P\equiv\{C_1,...,C_K\}$, and every set $A\in \cal{A}$:
$$\sum_i|\Psi(C_i\cap A)-\mu(C_i\cap A)|\leq \vp.$$
\end{description} 
\end{remark}
In the following proposition a result of countable additivity of $\Psi$ is proved, under suitable hypotheses.

\begin{proposition}\label{psisigmaadd}
Let $\mu:\cal{A} \to \erre^+_0$ be any increasing set function, satisfying $\mu(\emptyset)=0$, and assume that $\mu$ is  continuous from below.
 Then, if $\mu$ is integrable, the mapping $A\mapsto \Psi(A)=\int_A d\mu$ is countably additive.
\end{proposition}
\begin{proof}
Since $\Psi$ is additive and non-negative, it is sufficient to prove its continuity from below. So, let $(A_k)_n$ denote any increasing sequence in $\cal{A}$, with union $A$, and fix $\vp>0$. Then,
 there exists an $\vp/2$-approximated partition of $T$, denoted by $P_{\vp}:=\{E_1,...,E_n\}$. Thanks to the item \ref{fatti-p}.1, the partition $\{E_1\cap A_k, E_2\cap A_k,...,E_n\cap A_k\}$ is an $\vp$-approximated partition of $A_k$ for all $k$. Now, let $k_0$ be an integer large enough that 
$$\mu(E_i\cap A)\leq \mu(E_i\cap A_k)+\frac{\vp}{n}$$
for all $i\in \{ 1, \ldots,n\}$ and all $k\geq k_0$. Then
$$\Psi(A)\leq \sum_{i=1}^n\mu(A\cap E_i)\leq \sum_{i=1}^n\mu(A_k\cap E_i)+\vp\leq \Psi(A_k)+2\vp,$$
for all $k\geq k_0$. This implies that $\Psi(A)\leq \lim_k\Psi(A_k)$. The converse inequality is trivial, by monotonicity of $\Psi$.
\end{proof}
\vskip.3cm

In what follows, let $(X,\|\cdot \|)$ be a Banach
space with the origin 
$0$, $\mathcal{P}_{0}(X)$ the family of all nonvoid subsets of $X$,
and consider its subfamilies:\\
-$b(X)$  the family of nonvoid bounded subsets of $X$,\\
-$bf(X)$ the family of closed, bounded, nonvoid subsets of $X$, \\
-$bcf(X)$ the family of all bounded, closed, convex, nonvoid subsets of $X$, \\
-$ck(X), (cwk(X))$ the family of
all (weakly) compact and convex, nonvoid subsets of $X$.\\

Let $h(A,B)=\max \{e(A,B),e(B,A)\},$ for every  non empty subsets $A,B$ of $X$, 
where $$e(A,B)=\underset{x\in A}{\sup }\ d(x,B)$$ is \textit{the
excess} of $A$ over $B$ and $d$ is the metric induced by the norm of $X$.
On $bcf(X)$, $h$ becomes a metric, called the \textit{Hausdorff metric}.
For every $A\in \mathcal{P}_{0}(X)$, denote $$|A|=h(A,\{0\})=\sup_{x\in A} ||x||.$$

For every $A\subseteq X$, the symbols $\overline{A}, \oplus $ denote the norm closure of $A$ and
the norm closure of the Minkowsky addition respectively. \\
When the symbol $\sum$ is used for multifunctions this means also the norm closure of the addition.
For the properties of the Hausdorff metric $h$ see for example
\cite{cv,pgcsub,papalibro}. 
Some of these 
are now recalled:
\begin{proposition}\label{ex2.8}
 The following properties hold for every $\alpha
,\beta \in \mathbb{R},A,B,C,D \in b(X):$
\begin{enumerate}[\rm  \ref{ex2.8}.i)]
\item $h(\alpha A,\beta A)\leq |\alpha -\beta |\cdot |A|;$
\item $h(A\oplus B,C\oplus D)\leq h(A,C) + h(B,D);$
\item $h(A\oplus C,B\oplus C)\leq h(A,B);$
\item $h(A,B)\leq |A|+|B|;$
\end{enumerate}
\end{proposition}

\begin{lemma}\label{prop-h}
Fix $A,B,C \in b(X)$.
Then
\[ |h(A \oplus B, C) - h(A,C) | \leq | B|.\] 
\end{lemma}
\begin{proof}
It is
\begin{eqnarray*}
h(A \oplus B, C ) &=& h(A \oplus B, C \oplus \{ 0\}) \leq h(A,C) + |B |;\\
h(A, C) &\leq& h(A, A \oplus B) + h(A \oplus B,C) \leq |B| + h(A \oplus B,C),
\end{eqnarray*}
so the assertion follows.
\end{proof}
Recall that $ck(X)$ ($cwk(X)$) denotes the family of all convex and (weakly) compact subsets of  $X$. 
Thanks to the R{\aa}dstr\"{o}m embedding theorem,
$ck(X)$ ($cwk(X)$) endowed with the Hausdorff distance $h$ is a complete metric space that can be isometrically embedded
 into a Banach space:

\begin{lemma}\label{labu} {\rm (\cite[Theorems 5.6 and 5.7]{L1})}
Let $S=cwk(X)$ or $ck(X)$. Then there exist a compact Hausdorff space $\Omega$ and an isometry $j:S\to C(\Omega)$, endowed with the $\sup$-norm, such that 
\begin{enumerate}[\rm   (\ref{labu}.a)]
\item $j(\alpha A + \beta C) = \alpha j(A) + \beta j(C)$ for all $A,C \in S$ and $\alpha, \beta \in \mathbb{R}^+$,
\item $h(A,C) = \| j(A) - j(C) \|_{\infty}$ for every $A,C \in S$,
\item $j(S)$ is norm closed in $C(\Omega)$,
\item $j(\mbox{co}(A \cup C)) = \max \{j(A), j(C) \}$, for all $A,C \in S$.
\end{enumerate}
\end{lemma}

\begin{definition}\label{ex2.9} \rm 
A multifunction $F:T\rightarrow b(X)$
is called \textit{bounded} if there exists $M\in \mathbb{R}_0^{+}$
 such that 
$|F(t)|\leq M$, for every $t\in T. $
\end{definition}

\begin{definition}\label{ex2.10} \rm 
 Let $\upsilon :\mathcal{A}\rightarrow \mathbb{R}_0^+$ be a non-negative set function and let 
$F:T\rightarrow \mathcal{P}_{0}(X)$ be a multifunction.
\begin{enumerate}[\rm \ref{ex2.10}.i)]
\item $F$ is called $\upsilon $\textit{-totally-measurable} (\textit{on} $T)$
if for every $\varepsilon >0$ there exists a finite family $\{A_{i}\}_{i \in \{0, \ldots, n\}}$ 
of nonempty pairwise disjoint subsets of $T$, such that
$\bigcup\limits_{i=0}^{n}A_{i}=T$,
 $\upsilon(A_{0})<\varepsilon \;$ and
 $\sup\limits_{t,s\in A_{i}}h(F(t),F(s))=osc(F,A_{i})<\varepsilon ,\;$
for\ every $i \in \{1, \ldots, n\}.$
\item $F$ is called $\upsilon $\textit{-totally-measurable on }$B\in \mathcal{A
}$ if the restriction $f|_{B}$ of $f$ to $B$ is $\upsilon $-totally
measurable on $(B,\mathcal{A}_{B},\upsilon _{B})$, where $\mathcal{A}
_{B}=\{A\cap B;A\in \mathcal{A}\}$ and $\upsilon _{B}=\upsilon |_{\mathcal{A}
_{B}}$.
\end{enumerate}

\end{definition}
\section{A Birkhoff simple integral}\label{three}
In the sequel, suppose $\mu :\mathcal{A}\rightarrow \mathbb{R}_0^+$ is a
non-negative set function with $\mu (T)>0$
and the cardinality of $T$ is greater than $\aleph_0$.
 The Birkhoff multivalued integral that will be extended is the following:
\begin{definition}\rm \label{2.1cr2004} (\cite[Definition 2.1]{cr2004})
Let $F: T \to cwk(X)$ be a multifunction and
$\mu$ is a non negative countably additive measure 
defined on a $\sigma$-algebra  $\mathcal{A}$.
$F$ is Birkhoff integrable if the single valued function $j \circ F: T \to C(\Omega)$ is Birkhoff integrable.
(Observe that the Authors proved that the definition is indipendent of the embedding used, in the paper they used  $l_{\infty}(B_{X^*})$.)
\end{definition}
To this aim  a slight modification of  \cite[Definition 3.1]{cgsub1} is proposed:
\begin{definition}\label{ex3.5} \rm 
 Let $F:T\rightarrow b(X)$ be a
multifunction and let $\mathcal{C}$ be a non empty subfamily of  $b(X)$.  Let also $\mu:\mathcal{A}\to \mathbb{R}_0^+$ be a
non negative set function with $\mu(T)>0$, defined in a $\sigma$-algebra  $\mathcal{A}$. $F$
is called {\em Birkhoff simple integrable} in $\mathcal{C}$ on 
$T$ if there exists a closed set $I\in \mathcal{C} $ with the following property: for
every $\varepsilon >0$, there exist a countable partition
 $P_{\varepsilon }\subset {\mathcal C}$
of $T$ 
so that for every countable
partition $P=\{A_{n}\}_{n\in \mathbb{N}}$ of $T$, with $P\geq P_{\varepsilon
}$ and every $t_{n}\in A_{n},n\in \mathbb{N}$, one has 
\begin{eqnarray}\label{maxlim1}
 \limsup_n h\left(\sum_{i=1}^nF(t_i)\mu(A_i),I\right)\leq \vp.
\end{eqnarray}
The set $I$ is called the \textit{Birkhoff simple integral in $\mathcal{C}$ of }$F$ 
\textit{\ on }$T$ \textit{\ with respect to} $\mu $ and is denoted by $(B_s)\int_{T}Fd\mu .$
\\
The simple Birkhoff integrability/integral of $F$ on a set $A\in \mathcal{A}$ is defined in the usual 
 manner.
\end{definition}
\begin{remark}\rm 
Observe that, in \cite{cgsub1}, the definition does not make use of the $\limsup$ operation, but requires instead that for every $\varepsilon>0$ there exist a countable partition $P_{\varepsilon }$ and a suitable integer $n_{\varepsilon}$ such that 
$$h\left(\sum_{i=1}^nF(t_i)\mu(A_i),I\right)\leq \vp$$
holds, as soon as $P=\{A_{n}\}_{n\in \mathbb{N}}$ is finer than 
$P_{\varepsilon }$  and $n\geq n_{\vp}$.
So the Definition \ref{ex3.5} looks more intuitive, and at least formally weaker, since it avoids the condition of uniformity about the integer $n_{\vp}$.\\
\end{remark}

\begin{corollary}\label{jFsB}
Let $F : T \to cwk(X)$ be a multifunction. Then $F$ is Birkhoff simple integrable  in $cwk(X)$ if and only if $j \circ F$ is
Birkhoff simple integrable.
\end{corollary}
\begin{proof}
It follows directly from Lemma \ref{labu} since $j(cwk(X))$   is norm closed in $C (\Omega)$ and 
$j$ is norm preserving and additive,
namely
\[ h\left(\sum_{i=1}^n F(t_i)\mu(A_i),I\right) = \| \sum_{i=1}^n j \circ F(t_i) \mu(A_i) - j(I) \|_{\infty} \].
\end{proof}

Observe moreover that the Birkhoff simple integral is additive, in fact:
\begin{theorem}\label{ex3.7}
 If $\mu$ is finitely additive, and $F:T\rightarrow b(X)$  is Birkhoff simple integrable
with respect to  $\mathcal{C}$ 
 on a pair of disjoint sets  $A,B\in \mathcal{A}$, 
then $F$is Birkhoff simple integrable
in $\mathcal{C}$ on $A\cup B$ too and 
$$(B_s)\int_{A\cup B}Fd\mu = (B_s)\int_{A}Fd\mu \oplus (B_s)\int_{B}Fd\mu . $$
\end{theorem}
\begin{proof}
The proof works as in \cite[Theorem 4.1]{cgsub1}.
\end{proof}

In order to prove that Definition \ref{ex3.5} is an extension of
 Definition \ref{2.1cr2004}  when $\mu$ is a non-negative countably additive measure and $\mathcal{C}$ is $cwk(X)$ or $ck(X)$ a comparison with the Gould integral is needed. 
This result will be obtained for the case of finite or $\sigma$-finite
non negative, (countably) additive measure.
\subsection{The case of  finite measure $\mu$}\label{three.1}
The first step is to compare the Gould and simple Birkhoff integrability in the case of $\mu$ finite. 
 The notion of Gould integral has been introduced in Definition  \ref{ex3.1}. 

\noindent In case $F$ takes values in one of the spaces $ck(X)$ or $cwk(X)$, thanks to completeness of these spaces with respect to the Hausdorff distance, the Gould integral, if it exists, is in the same space $ck(X)$ or $cwk(X)$ respectively. Indeed
\begin{proposition}\label{integraleck}
Assume that $F:T\to cwk(X)$ is Gould integrable, then the integral $I$ is in $cwk(X)$. The same holds for $ck(X)$-valued mappings.
\end{proposition}
\begin{proof}
Only the case of $cwk(X)$-valued functions will be discussed, the other one is quite similar and is contained in \cite[Remark 3.7]{pgcsub}.\\
Let $I$ denote the Gould integral of $F$, and fix arbitrarily $k\in \mathbb{N}$. Then there exists a finite partition 
$P_k$ of $T$, $P_k:=\{E_1^k, \ldots, E^k_{n_k}\}$, such that, for every finer partition 
$P:=\{B_1,...,B_m\}$ and every choice of points $t_j\in B_j,j=1,...,m$, it holds
$$h(\sum_{j=1}^mF(t_j)\mu(B_j),I)\leq \frac{1}{k}.$$
In particular, as soon as points $\tau_1\in E^k_1,...,\tau_{n_k}\in E^k_{n_k}$ are fixed, one has
$$h(\sum_{i=1}^{n_k}F(\tau_i)\mu(E^k_i),I)\leq \frac{1}{k}.$$
Now, setting $s_k:=\sum_{i=1}^{n_k}F(\tau_i)\mu(E^k_i)$, the elements $s_k$ are in $cwk(X)$, and clearly they form a Cauchy sequence with respect to the Hausdorff distance. So, by completeness of $(cwk(X),h)$, there exists an element $I_0\in cwk(X)$ such that $\lim_k s_k=I_0$ in that space. Hence, for each $\vp>0$ there exists an integer $k_{\vp}$ such that $h(s_k,I_0)\leq \vp$ for every $k\geq k_{\vp}$.
\\
The element $I_0$ is the Gould integral of $F$. 
Indeed, as soon as $\vp>0$ is fixed, choose any integer $k^*$ larger than $k_{\vp}\vee\frac{1}{\vp}$. Then, for every partition $P:=\{B_1,...,B_m\}$ finer than $P_{k^*}:=\{E_1^{k^*}, \ldots, E^{k^*}_{n_{k^*}}\}$,
 and every choice of points $t_j\in B_j,j=1,...,m$, it holds
\begin{eqnarray*}
h(\sum_{j=1}^mF(t_j)\mu(B_j),I_0) &\leq& h(\sum_{j=1}^mF(t_j)\mu(B_j),
I) + h(I,s_{k^*})+h(s_{k^*},I_0)\leq \\ &\leq&
 \frac{2}{k^*}+\vp\leq 3\vp.
\end{eqnarray*}
This clearly suffices to prove the assertion.

\end{proof}

\begin{remark}\rm \label{jFG}
Observe that, when the mapping $F$ is single-valued, this definition is the same as 
{\em RLF-integrability} of \cite[Definition 13]{pot2007}. Moreover, the notion of  RLF-integrability can be adopted for the case of $cwk(X)$- or $ck(X)$-valued maps, 
as in Corollary \ref{jFsB},
thanks to the R{\aa}dstr\"{o}m embedding theorem (Lemma \ref{labu})
since, by Proposition \ref{integraleck},
the integral lives in the same hyperspace, $cwk(X)$ or $ck(X)$. 
Indeed, since the embedding $j$ of $ck(X)$ or $cwk(X)$  is an isometry, it is clear that
$$h(\sum_{i=1}^n F(t_{i})\mu (A_{i}),I)=\|\sum_{i=1}^n j(F(t_{i}))\mu (A_{i})-j(I)\|_{\infty}$$ for all partitions $\{A_n:n\in \mathbb{N}\}$ and choices $t_n\in A_n$, so it follows that any $cwk(X)$- or $ck(X)$-valued mapping $F$ is Gould-integrable if and only if $j(F)$ is, as a single-valued function (as was done in  Corollary \ref{jFsB} for the Birkhoff simple integral).

\end{remark}

\begin{theorem}\label{ex4.1}
Suppose $\mathcal{A}$ is a $\sigma$-algebra, $\mu :\mathcal{A}\rightarrow \lbrack 0,\infty )$ 
is a countably additive measure and $F:T\rightarrow b(X)$
is a bounded multifunction. Then $F$ is Birkhoff simple integrable if and only if it is $G$-integrable and 
$$(B_s)\int_{T}Fd\mu =(G)\int_{T}Fd\mu .$$
\end{theorem}
\begin{proof} 
Assume that $F$ is simple-Birkoff integrable, with integral $I$, and fix $\vp>0$.
 Then there exists a countable partition $P^1:=\{E_j:j\in \mathbb{N} \}$ 
such that, for every countable partition $P^2$ finer than $P^1$, $P^2=\{B_j,j\in \mathbb{N}\}$, and every choice of the points $t_j\in B_j$ it is
\begin{eqnarray}\label{limsup1}
\limsup_nh(\sum_{i=1}^nF(t_i)\mu(B_i),I)\leq \frac{\vp}{4}.
\end{eqnarray}
Let $M$ be any majorant for $\sup_{t\in T}h(F(t),\{0\})$.
Then, let $N$ be any integer such that $\sum_{j> N}\mu(E_j)\leq \frac{\vp}{4M}$, and denote by $A$  the union of all the sets $E_j,j>N$. 
Clearly, for every (countable) partition of $A$, say $\{A_i,\ i\in \mathbb{N}\}$, and every choice of points $\tau_i\in A_i$ for all $i\in \mathbb{N}$, one has
\begin{eqnarray}
\label{sommeresti}  && \sum_i h(F(\tau_i)\mu(A_i),\{0\})\leq \frac{\vp}{4};\\
\label{segue} && h(\sum_{i=1}^NF(t_i)\mu(E_i),I)\leq \frac{\vp}{2}.
\end{eqnarray}
In particular, for $E_i, i \geq N$ it holds
\begin{eqnarray}\label{sommedue}
h(\sum_{i=1}^{N+n}F(t_i)\mu(E_i),\sum_{i=1}^NF(t_i)\mu(E_i))\leq \frac{\vp}{4}\end{eqnarray}
for every integer $n$, and every choice of the points $t_i\in E_i$, since the elements $E_j$, $j>N$, form a partition of $A$.\\
Let $P_{\vp}:= \{E_1, E_2, \ldots, E_N, A \}$.
 For every finite partition $P'':=\{G_1,...,G_l\}$ of $T\setminus A$, finer than $\{E_1,...,E_N\}$, , $i=1,...,l$,
consider the partition $P^*:=\{G_1,...,G_l,E_{j}, j \geq N+1\}=\{C_i, i \in \mathbb{N}\}$, finer than $P^1$ and
any choice of the points $t''_i\in G_i, \tau_k \in E_k, k>N$:
then there exists $n_1 \geq \max \{N, l \}$ such that for every $n \geq n_1$ it is 
\[ h(\sum_{i=1}^{l} F(t''_i)\mu(G_i) \oplus \sum_{i=l+1}^{n} F(\tau_i)\mu(E_i) ,I)\leq \frac{\vp}{2}
\]
Then, by Lemma \ref{prop-h} and (\ref{sommeresti}) it is

\begin{eqnarray}\label{sommefinite2}
h(\sum_{i=1}^l F(t''_i)\mu(G_i),I)\leq \frac{\vp}{2} + \frac{\vp}{4}.
\end{eqnarray}
So, if $P'$ is any finite partition, finer than $P_{\vp}$, $P'=\{H_1,..,H_u\}$, for every choice of the points $\tau_i\in H_i$, $i=1,...,u$, one clearly has
$$\sum_{i=1}^uF(\tau_i)\mu(H_i)=\sum_{H_i\subset T\setminus A} F(\tau_i)\mu(H_i) \oplus
\sum_{H_i\subset A}F(\tau_i)\mu(H_i),$$
and
$$h(\sum_{H_i\subset T\setminus A}F(\tau_i)\mu(H_i),I)\leq  \frac{3}{4}\vp$$
by (\ref{sommefinite2})
and
$$h(\sum_{H_i\subset A}F(\tau_i)\mu(H_i),\{0\})\leq \frac{\vp}{4}$$
by (\ref{sommeresti}). Then clearly one can conclude:
$$h(\sum_{i=1}^uF(\tau_i)\mu(H_i),I)\leq \vp. $$

Now, the converse implication will be proved. So, assume that $F$ is Gould-integrable, with integral $I$: then for every $\vp>0$ a finite partition $P_{\vp}$ exists, such that, for every finite finer partition $P':=\{B_1,...,B_m\}$, and every choice of the points $t_i\in B_i$, $i=1,...,m$, one has
$$h(\sum_i F(t_i)\mu(B_i),I)\leq \vp/2.$$

The partition $P_{\vp}$ is suitable also for simple Birkhoff integrability. Indeed, denote $P_{\vp}:=\{E_1,...,E_n\}$, and pick any countable finer partition $P'$, $P':=\{B_1,B_2,...\}$, together with a choice of tags $\tau_j\in B_j$, $j\in \mathbb{N}$. By countable additivity, there exists an integer $j_0$ large enough that $\mu(\bigcup_{j>j_0}B_j)\leq \frac{\vp}{4M}$.
Therefore, for every choice of the points $t_j\in B_j$ one has
\begin{eqnarray}\label{restisomme}
\sum_{j>j_0}h(F(t_j)\mu(B_j),\{0\})\leq \frac{\vp}{4}.
\end{eqnarray}
Denote by $Q$ the union of all the sets $B_j$, with $j>j_0$.
Consider now the following finite partition $P''$, finer than $P_{\vp}$: $P''$ is  obtained by taking for each index $i$ the elements of the type $B_j,\ j\leq j_0$, that are contained in $E_i$, (with the corresponding  tag $\tau_j$) and also $E_i\cap Q$, with any tag. Now, since $P''$ is finer than $P_{\vp}$,
$$h(\Sigma(F,P''),I)\leq \vp/2;$$
then, from (\ref{restisomme}) and Lemma \ref{prop-h} one clearly has
$$h(\sum_{j=1}^JF(\tau_j)\mu(B_j),I)\leq \frac{\vp}{2}+\frac{\vp}{4}\leq \vp,$$
as soon as $J\geq j_0$. This also means that
$$\limsup_Jh(\sum_{j=1}^JF(\tau_j)\mu(B_j),I)\leq \vp$$
and therefore the assertion.
\end{proof} 
In the next result, another relationship between Gould and simple Birkhoff 
integrabilities on atoms is presented. This is done in weaker hypothesis, i.e. just assuming 
that $\mu $ is $\sigma$-null-null-additive and monotone.
\begin{theorem}\label{ex4.3}
Suppose $F:T\rightarrow b(X)$  is fixed,
$\mu :\mathcal{A}\rightarrow \mathbb{R}^{+}_0$
is $\sigma$-null-null-additive and monotone and $A\in \mathcal{A}$
 is an atom of $\mu $. Then  $F$ is Birkhoff simple
integrable on $A$ if and only if  $F$ is Gould integrable on $A$
 and moreover
$$(G)\int_{A}Fd\mu =(B_s)\int_{A}Fd\mu .$$
\end{theorem}
\begin{proof} Let $\varepsilon >0$ be arbitrary and denote 
$I=(B_s)\int_{A}Fd\mu .$ There exist a countable partition $P_{\varepsilon}:=\{B_{n}\}_{n\in \mathbb{N}}$ of $A$ and 
an integer $n_0\in\mathbb{N}$ 
so that for every $t_{n}\in B_{n}$ 
one has
\begin{eqnarray}\label{15}
h(\sum_{k=1}^n F(t_{k})\mu (B_{k}),I)<\varepsilon , \mbox{\, for every\,} n\geq n_0.
\end{eqnarray}
Since $A=\underset{n=1}{\overset{\infty }{\cup }}B_{n}$ is an atom of $\mu ,$
by Proposition \ref{ex2.6}.ii) it can be assumed
without  loss of generality that 
$\mu (B_{1})>0,B_{1}$ is an atom of $\mu $, $\mu (B_{1})=\mu (A)$ and
$\mu(B_{n})=0$, for every $n\geq 2.$ \\
Therefore, 
$\underset{k=1}{\overset{
n}{\sum }}F(t_{k})\mu (B_{k})=F(t_{1})\mu (B_{1})$, so (\ref{15}) becomes 
\begin{eqnarray}\label{16}
h(F(t_{1})\mu (B_{1}),I) <\varepsilon .
\end{eqnarray}
So let  $\widetilde{P}_{\varepsilon }=\{B_{1},\underset{n=2}
{\overset{
\infty }{\cup }}B_{n}\}$ be a finite partition of $A$.
\\
Consider an arbitrary partition $\widetilde{P}=\{C_{i}\}_{i=1}^m
$ of $A$, so that $\widetilde{P}\geq \widetilde{P}_{\varepsilon }$ and
arbitrary $t_{i}\in C_{i},i=1 \ldots m.$\\
 Then $m\geq 2$ and
 $B_{1}=\underset{i=1}{\overset{l}{\cup }}C_{i}$, $\underset{n=2}{
\overset{\infty }{\cup }}B_{n}=\underset{i=l+1}{\overset{m}{\cup }}C_{i}$,
for some $l\in \{1,\ldots,m-1\}.$
\\
Since $\mu $ is $\sigma$-\textit{-}null-null-additive, then $\mu (\cup_{n=2}^{\infty} B_{n})=0$, so, by the monotonicity of $\mu $
it results $\mu (C_{i})=0$, for every $i\in \{l+1, \ldots, m\}.$ Then, 
\begin{eqnarray}\label{17}
\sum_{i=1}^m F(t_{i})\mu (C_{i})=\sum_{i=1}^l F(t_{i})\mu (C_{i}).
\end{eqnarray}
On the other hand, because $B_{1}$ is an atom of $\mu $ and $\{C_{i}\}_{i=1, \ldots, l}$ is a partition of $B_{1}$, by 
Proposition \ref{ex2.6}.i), 
it is possible to assume, without loss of generality, 
that $\mu (C_{1})=\mu (B_{1})$ and $\mu (C_{i})=0$, for every $i\in\{2,\ldots, l\}.$\\ 
From $(\ref{17})$ it follows 
\begin{eqnarray*}
\underset{i=1}{\overset{m}{\sum }}F(t_{i})\mu (C_{i})=F(t_{1})\mu (B_{1}),
\end{eqnarray*}
which yields from $(\ref{16})$ that 
\begin{eqnarray*}
h(\underset{i=1}{\overset{m}{\sum }}F(t_{i})\mu (C_{i}),I)=h(F(t_{1})\mu
(B_{1}),I)<\varepsilon .
\end{eqnarray*}
So, $F$ is Gould $m$-integrable on $A$ and 
$$(G)\int_{A}Fd\mu
=(B_s)\int_{A}Fd\mu .$$

Viceversa suppose that $F$ is Gould integrable on $A$, namely 
 there exists a set $I_A \in bf(X)$ such that for every 
$\varepsilon >0$, there is a finite partition $P_{\varepsilon, A }$ of 
$A$ such that for every partition of $A$, $P:=\{A_{i}\}_{i=1}^{n}$, with $P\geq P_{\varepsilon,A }$ and
every $t_{i}\in A_{i},i\in \{1,2,\ldots,n\}$, it holds
$$h(\sum_{i=1}^n F(t_{i})\mu (A_{i}),I_A)<\varepsilon.$$
As before, since $A$ is an atom by Proposition \ref{ex2.6}.i) it can be assumed
without  loss of generality that $A_1$ is an atom,
 $\mu (A_{1})=\mu (A)$ and $\mu(A_{n})=0$, for every $n\geq 2.$
Therefore, 
$\sum_{k=1}^n F(t_{k})\mu (A_{k})=F(t_{1})\mu (A_{1})$, so the previous inequality becomes 
\begin{eqnarray*}\label{16bis}
h(F(t_{1})\mu (A_{1}),I_A) <\varepsilon .
\end{eqnarray*}
Consider now an arbitrary countable partition $P^c_{\varepsilon, A}$ of $A$ 
$P^c_{\varepsilon, A} \geq \{A_1, \cup_{j \geq 2} A_j\}$.
Then, for every $P \geq P^c_{\varepsilon, A}$, let $\mathbb{M}_p$ the set of indexes of elements of $P$  which cover $A_1$. Since $A_1$ is an atom, one can assume, without loss of generality, that $\mathbb{M}_p$ consists of just one element $m_0$.
So,
$$\sup_{k\geq m_0}h\left(\sum_{i=1}^kF(t_i)\mu(A_i),I\right)\leq \vp,$$
and this proves that $F$  is Birkhoff simple
integrable on $A$.
\end{proof}
The two previous theorems extend analogous results  in \cite{cg2014}  obtained by some of the Authors when $F$ is a scalar   Birkhoff integrable function.
\\

Now,  in order to prove that Birkhoff simple integrability extends 
Definition \ref{2.1cr2004},
the  case $cwk(X)$-valued  is considered and  a single valued result is stated:
\begin{theorem}\label{partfinite}
Let $(T,\mathcal{A},\mu)$ be a finite countably additive measure space, and $f:T\to X$ be any bounded function. Then the following are equivalent:
\begin{enumerate}[\bf  \rm \ref{partfinite}.1)]
\item $f$ is Gould-integrable;
\item $f$ is Birkhoff-integrable.
\end{enumerate}
\end{theorem}
\begin{proof}
It is a consequence of \cite[Theorems 15 and 16]{pot2007}, when dealing with single-valued functions, see also \cite[Theorem 5.1]{cg2014}.
\end{proof}

\begin{remark} \rm \label{bgmulti}
Clearly, the same equivalence is valid also for $ck(X)$-or $cwk(X)$-valued  bounded multifuctions $F$ 
thanks to Proposition \ref{integraleck} and Lemma \ref{labu} since $j \circ F$ is Gould integrable if and only if $F$ is by Remark \ref{jFG}.
\end{remark}

An easy consequence of the Henstock Lemma \ref{fatti-p}.2) is the following result, that will be useful later.

\begin{proposition}\label{mupsiwb}
Let $F:T\to cwk(X)$ be any bounded multifunction, and assume
 that $\mu$ is finite, non-negative, monotone and integrable, with integral function $\Psi$. Then $F$ is simple-Birkhoff (resp. Gould) integrable with respect to $\mu$ if and only if it is simple-Birkhoff (resp. Gould) integrable  with respect to $\Psi$, with the same integral.
\end{proposition}
\begin{proof}
The proof will be given only for the simple Birkhoff integrability: the other case is quite similar.

Let $M$ denote any positive upper bound for $\sup_{t\in T}|F(t)|$. Fix $\vp$ and let $P_{\vp}$ be any $\vp$-approximated partition of $T$ for the Henstock Lemma (see \ref{fatti-p}.2). Now, assume that $F$ is simple-Birkhoff integrable with respect to $\mu$ (with integral $J$), and fix any countable partition $P'_{\vp}$, finer than $P_{\vp}$ and such that, for every finer partition $P'':=\{B_1,B_2,...\}$ and every choice of the tags $t_i\in B_i$, it holds
$$\limsup_nh(\sum_{i=1}^nF(t_i)\mu(B_i),J)\leq \vp.$$ 
Now, since $P''$ is finer than $P_{\vp}$, thanks to the Henstock Lemma \ref{fatti-p}.2), for every $n$ it holds
$$\sum_{i=1}^nh(F(t_i)\mu(B_i),F(t_i)\Psi(B_i))\leq M \vp$$
(indeed, the pairwise disjoint sets $B_1,...,B_n$ can be seen as part of a finite partition finer than $P_{\vp}$).
So, 
$$\limsup_nh(\sum_{i=1}^n F(t_i)\Psi(B_i),J)\leq (1+M)\vp.$$
This suffices to prove that $F$ is simple Birkhoff integrable with respect to $\Psi$.
\\
A perfectly similar argument shows the converse implication.
\end{proof}

Now the main result of this subsection can be achieved as a 
 consequence of Theorems \ref{partfinite} and \ref{ex4.1}. Namely
at least for $cwk(X)$- or $ck(X)$-valued bounded functions and when $\mu$ is finite and countably additive, there is equivalence among Birkhoff, simple Birkhoff and Gould integrability.\\

\begin{theorem}\label{general}
Let $\mu:T\to \erre^+_0$ be any finite monotone, integrable set function, continuous from below. Then, for every bounded function $F:T\to cwk(X)$, simple Birkhoff integrability with respect to $\mu$ is equivalent to Birkhoff and Gould integrability.
\end{theorem}
\begin{proof}
Let $\Psi$ denote the integral function of $\mu$. Since $F$ is bounded, Gould or (simple)-Birkhoff integrability of $F$ with respect to $\mu$ is equivalent to the same type of integrability with respect to $\Psi$ thanks to Propositions \ref{mupsiwb}. But $\Psi$ is countably additive by Proposition \ref{psisigmaadd}, so Theorem \ref{ex4.1} ensures that $F$ is (simple)-Birkhoff integrable with respect to $\Psi$ (i.e. $\mu$) if and only if it is Gould integrable with respect to $\Psi$ (i.e. $\mu$).
\end{proof}
\subsection{The case of $\sigma$-finite measure $\mu$}\label{three.2}
 In this section, a comparison will be obtained between simple-Birkhoff and  
Birkhoff integrability, in a more general setting, i.e. assuming that the measure $\mu$ is $\sigma$-finite, and also without requiring
 boundedness of $F$.  
Let's begin with a single valued result:
\begin{theorem}\label{versopettis}
Let $(T,\mathcal{A},\mu)$ be any $\sigma$-finite measure space,
 with $\mu$ countably additive and non-negative. If $f:T\to X$ is simple-Birkhoff integrable then $f$ is Pettis integrable.
\end{theorem}
\begin{proof}
Assume that $f$ is simple-Birkhoff integrable, with integral $x$. Then, for each $\vp>0$ there exists a partition $P_{\vp}:=\{E_1,...,E_n,...\}$ of $T$ such that, for every finer partition $P:=\{B_1,...,B_k,...\}$ it holds:
\begin{eqnarray}\label{sommelimitate}
\limsup_k\|\sum_{i=1}^kf(t_i)\mu(B_i)-x\|\leq \vp,
\end{eqnarray}
for every choice of the points $t_k\in B_k$, $k\in \mathbb{N}$.
Then clearly, for each element $x^*\in X^*$, the mapping 
$\langle  x^*, f\rangle$ is simple-Birkhoff integrable, with integral $\langle x^*,x\rangle$.
\\
Now, the Birkhoff integrability of   $\langle x^*,f\rangle$ will be proved for every $x^*$: to this aim, it will be sufficient to prove that, for each fixed $x^*$, all the series $$\sum_k\langle x^*,f(t_k)\rangle\mu(B_k)$$
 are absolutely convergent, as soon as  $P:=\{B_1,...,B_k,...\}$ is any finer partition of $P_{\vp}$ and for every choice $t_k\in B_k$, $k\in \mathbb{N}$. By contradiction, assume that one of these real-valued series is not absolutely convergent: then it is possible to rearrange its terms in such a way that the series diverge, and this would contradict the condition (\ref{sommelimitate}). 
So, the real-valued mapping $\langle x^*,f\rangle$ is Birkhoff-integrable, and therefore in $L^1$, for every element $x^*$ of the dual space $X^*$. Clearly, it is also
$$\int_T \langle x^*,f\rangle d\mu= \langle x^*,x\rangle$$
for all $x^*\in X^*$,
which proves Pettis integrability of $f$.
\end{proof}
Of course, this Theorem has an application to the multivalued case:
\begin{corollary}\label{jFP}
If $F:T\to cwk(X)$ is simple-Birkhoff integrable,
then $F$ and $j \circ F$ are Pettis integrable.
\end{corollary}
\begin{proof}
By Theorem \ref{versopettis} and thanks to the R{\aa}dstr\"{o}m embedding Theorem (Lemma  \ref{labu}), if $F$ is simple Birkhoff integrable
 so is $j(F)$ in view of Corollary \ref{jFsB}, where $j$ denotes the R{\aa}dstr\"{o}m embedding of $cwk(X)$, and therefore $j(F)$ is Pettis integrable. From this, and using the existence of the integral in $cwk(X)$, it follows also the Pettis integrability of $F$, proceeding like in the proof of \cite[Proposition 3.5]{cr2004} without separability assumption.\\
\end{proof}

Now, assuming that $(T,\mathcal{A},\mu)$ is a $\sigma$-finite quasi-Radon space, we shall prove that any bounded simple-Birkhoff integrable function is also Birkhoff integrable.
\begin{theorem}\label{mainequiv}
Let $(T,\mathcal{A},\mu)$ be a $\sigma$-finite 
measure space, 
If $F:T\to cwk(X)$ is a bounded simple-Birkhoff integrable mapping, then $F$ is also Birkhoff integrable.
\end{theorem}
\begin{proof}\
Assume that 
$T$ is a $\sigma$-finite measure space. Then $T$ is the union of a sequence $(A_n)_n$ of sets of finite measure, on which $F$ is bounded.
 Now, if $F$ is simple-Birkhoff integrable in each $A_n$, it is also Birkhoff integrable there, thanks to Theorems \ref{ex4.1} and \ref{partfinite}, namely $j \circ F$ is Birkhoff integrable on each $A_n$.
Moreover,  thanks to Corollary \ref{jFP},
one can deduce that $j \circ F$ is Pettis integrable in $T$. So, the condition $(ii)$ of \cite[Lemma 3.2]{cr2005} is satisfied, and therefore one can conclude that $j\circ F$ is Birkhoff integrable by virtue of that Lemma:
observe only that, though that Lemma is stated in the setting of a finite measure space, the proof of the crucial implication $(ii)\Rightarrow (i)$ does not make use of this condition, so the conclusion holds also in this case.
\end{proof}

In order to establish a more general result of equivalence, a Lemma is needed, similar to \cite[Lemma 5]{pot2007}.
\begin{lemma}\label{flimitata}
Let $(T,\mathcal{A},\mu)$ be a $\sigma$-finite measure space, with $\mu$ non-negative and countably additive, and let $f:T\to X$ be any mapping.
Assume that there exists a countable partition $P:=\{E_1,E_2,...,E_n,...\}$ and a positive real number $M$ such that
\begin{eqnarray}\label{maxlim}
\limsup_n\|\sum_{i=1}^n(f(t_i)-f(t'_i))\mu(E_i)\|\leq M,\end{eqnarray}
for every choice of $t_i$ and $t'_i$ in $E_i,\ i\in \mathbb{N}$. Then the mapping $f$ is bounded in every set $E_n$ such that $\mu(E_n)\neq 0$. 
\end{lemma} 
\begin{proof}
Fix  $n$ such that $\mu(E_n)\neq 0$ and fix any two points $t_n$ and $t'_n$ in $E_n$. Choose also arbitrarily a point $\tau_j$  in every set $E_j$, $j\neq n$. Then, for $N$ sufficiently large, $N>n$, it is
$$\|\sum_{1\leq j\leq N, \ j\neq n}(f(\tau_j)-f(\tau_j))\mu(E_i)+(f(t_n)-f(t'_n))\mu(E_n)\|\leq M$$
thanks to (\ref{maxlim}), from which it follows clearly
$$\|f(t_n)-f(t'_n)\|\mu(E_n)\leq M,$$
i.e. 
$$\|f(t_n)-f(t'_n)\|\leq \frac{M}{\mu(E_n)},$$
and this concludes the proof, by arbitrariness of $t_n$ and $t'_n$.
\end{proof}
Now, the final result of equivalence can be stated.
\begin{theorem}\label{sBvsB}
Let $(T,\mathcal{A},\mu)$ be a $\sigma$-finite measure space, with $\mu$ non-negative and countably additive, and let $f:T\to X$ be any mapping.
If $f$ is simple-Birkhoff integrable, then it is Birkhoff integrable.
\end{theorem}
\begin{proof}
Since $f$ is simple-Birkhoff integrable, there exists a partition $$P_1:=\{E_1,E_2,...,E_n...\}$$ such that, for every choice of points $t_n\in E_n$ one has
$$\limsup_n\|\sum_{i=1}^nf(t_i)\mu(E_i)\|\leq M$$
for a suitable positive real number $M$ (it suffices to take $\vp=1$ in the definition of simple-Birkhoff integrability). Then obviously also the condition (\ref{maxlim}) is satisfied. So, by Lemma \ref{flimitata}, $f$ is bounded in each set $E_n$ such that $\mu(E_n)\neq 0$, and, since $\mu(E_n)<+\infty$ for all $n$, $f$ turns out to be Birkhoff integrable in each set $E_n$ of the partition $P_1$. Moreover, $f$ is Pettis integrable by Theorem \ref{versopettis}, so the result of  \cite[Lemma 3.2]{cr2005} can be applied also in this case, and $f$ is Birkhoff integrable.
\end{proof}

Of course, the result of the last Theorem holds true in the following case:
\begin{corollary}
Let $F: T \to cwk(X) (ck(X))$. If $F$ is  is simple-Birkhoff integrable, then it is Birkhoff integrable. 
\end{corollary}
\begin{proof}
It follows directly from  R{\aa}dstr\"{o}m embedding Theorem (Lemma \ref{labu}) and Theorem \ref{sBvsB}.
\end{proof}
\section{A comparison with the Mc Shane multivalued integral}\label{four}
\noindent
So far, it has been proved that the definition of simple-Birkhoff integrability is an extension of the usual Birkhoff integrability as introduced in \cite{cr2004} for $cwk(X)$-valued mappings. Another interesting property of simple-Birkhoff integrability is that it implies a corresponding notion of Mc Shane integrability (see for example \cite{bms2011}) for $b(X)$-valued functions.\\

Assume that  $\mathscr{T}$ is a topology on $T$
 making $(T,\mathscr{T}, {\cal A}, \mu)$ a $\sigma$-finite quasi-Radon
 measure space which is \em outer regular, \rm namely such that
 $$\mu(B)=\inf \{ \mu(G):B \subseteq G \in \mathscr{T} \}\quad \mbox{for \, all  \,}
\, B \in {\cal A} .$$
\noindent A \em generalized Mc Shane partition \rm $\Pi$ of $T$
(\cite[Definitions 1A]{F2}) is a disjoint sequence $(E_i, t_i)_{i \in \mathbb{N}}$
of measurable sets of finite measure, with $t_i \in T$ for every $i \in \mathbb{N}$ and
$\mu(T \setminus \bigcup_i E_i)=0$. \\
A generalized Mc Shane partition  $(E_i, t_i)_{i \in \mathbb{N}}$ is said to be {\em univocally tagged} (in short, u.t.)
if each tag $t_i$ corresponds to just one of the sets $E_i$.
\\
If, moreover, $t_i\in E_i$ for all $i$, then the partition is said to be of the {\em Henstock} type.
A \em gauge \rm on $T$ is a function $\Delta: T \rightarrow \mathscr{T}$ such that $s \in \Delta(s)$ for every $s \in T$.
A generalized Mc Shane partition $(E_i,t_i)_i$ is \em  $\Delta$-fine \rm if $E_i \subset \Delta(t_i)$ for every $i \in \mathbb{N}$.\\
From now on, let ${\cal P}_s$ be the class of all generalized Mc Shane partitions of $T$
and by ${\cal P}_{\Delta}$  the set of all \mbox{$\Delta$}-fine elements of ${\cal P}_s$.\\

\begin{definition}\label{robanuova}\rm
Let $F:T\to b(X)$ be any multivalued mapping. $F$ is said to be {\em Mc Shane-integrable} if there exists an element $I\in bf(X)$ such that, for every $\vp>0$ there exists a gauge $\Delta$ on $T$ with the property that, for every $\Delta$-fine Mc Shane partition $P:=\{(t_i,A_i),i\in \mathbb{N}\}$ the following holds:
$$\limsup_nh(\sum_{i=1}^nF(t_i)\mu(A_i),I)\leq \vp.$$ 
In this case, $I$ is said to be the {\em Mc Shane integral} of $F$.
A formally more general notion is obtained, called {\em u.t. Mc Shane integrability}, if in the above requirements only u.t. partitions are involved.
\end{definition}
Observe that this notion of Mc Shane integral includes the one given  in \cite{bs2004}  for $cwk(X)$- or $ck(X)$-valued mappings.

\begin{lemma} \rm \label{nota2}
Assume that every point in $T$ is contained in an open set of finite measure, and let $F:T\to b(X)$ be any bounded function. If $F$ is u.t. Mc Shane integrable, with integral $I$, then it is also McShane integrable.
\end{lemma}
\begin{proof}
Let $M$ denote any positive constant, dominating $|F|$, and fix arbitrarily $\vp >0$. Then there exists a gage $\Delta$ fitting the condition of u.t. Mc  Shane integrability with respect to $\vp 3^{-1}$, and, without loss of generality, one can assume that $\mu(\Delta(t))<+\infty$ for all $t\in \Omega$.
\\
Now, let $P\equiv\{(B_j,w_j):j \in \mathbb{N}\}$ be any $\Delta$-fine Mc Shane partition.
For each tag $w_k$, let us denote by $H_k$ the union of those sets $B_j$ that are associated to the same tag $w_k$. Since $\mu$ is countably additive and $\mu(H_k)\leq \mu(\Delta(w_k))<+\infty$, it is possible to find a finite number of these sets $B_j$, denoted by
$$B_1^k,...,B_{j(k)}^k,$$
such that
$$\mu\left(H_k\setminus (\bigcup_{i=1}^{j(k)} B_i^k )\right)\leq \frac{\vp}{4^kM}.$$
Hence, setting $Q:=\bigcup_k\left(H_k\setminus (\bigcup_{i=1}^{j(k)} B_i^k )\right)$, one has $\mu(Q)\leq \vp (3M)^{-1}.$
Now, observe that $\{H_k,w_k):k\in \mathbb{N}\}$ is a $\Delta$-fine u.t. Mc Shane partition, so it satisfies
$$\limsup_nh\left(\sum_{k=1}^nF(w_k)\mu(H_k),I\right)\leq \frac{\vp}{3}.$$
Hence, there exists an integer $N$ such that
\begin{eqnarray}{\label{uniq}}
h\left(\sum_{k=1}^{N'}F(w_k)\mu(H_k),I\right)\leq \frac{\vp}{3}
\end{eqnarray}
for all $N'\geq N$.
In correspondence with $N$, there exists a (larger) integer $N_1$ such that the family $\{B_j:j\leq N_1\}$ includes all the sets of the type
$$B_i^k, i=1,...,j(k), \ \ k=1,...,N.$$
Remark that, if $j>N$, then $B_j\subset Q$.
So, for each integer $N_1'\geq N_1$ one has
\begin{eqnarray}\label{riduz}
h\left(\sum_{j=1}^{N_1'}F(w_j)\mu(B_j),\sum_{k=1}^N\sum_{i=1}^{j(k)}F(w_k)\mu(B_i^k)\right)
\leq \frac{\vp}{3}
\end{eqnarray}
since each set of the first summand which is not in the second one is included in $Q$, and $\mu(Q)\leq \vp (3M)^{-1}$.
For the same reason, one sees that
\begin{eqnarray}\label{confronto}
h\left(\sum_{k=1}^NF(w_k)\mu(H_k),\sum_{k=1}^N\sum_{i=1}^{j(k)}F(w_k)\mu(B_i^k)\right)\leq \frac{\vp}{3}.
\end{eqnarray}
Now, taking into account of (\ref{riduz}), (\ref{confronto}) and (\ref{uniq}), it follows
\begin{eqnarray*}
&& h\left(\sum_{j=1}^{N_1'}F(w_j)\mu(B_j),I\right)\leq 
h\left(\sum_{j=1}^{N_1'}F(w_j)\mu(B_j),\sum_{k=1}^N\sum_{i=1}^{j(k)}F(w_k)\mu(B_i^k)\right)+\\
&+&
h\left(\sum_{k=1}^NF(w_k)\mu(H_k),\sum_{k=1}^N\sum_{i=1}^{j(k)}F(w_k)\mu(B_i^k)\right)+
h\left(\sum_{k=1}^{N}F(w_k)\mu(H_k),I\right)\leq \vp.
\end{eqnarray*}
Since $N_1'\geq N_1$ was arbitrary, this show that
$$\limsup_nh\left(\sum_{j=1}^nF(w_j)\mu(B_j),I\right)\leq \vp$$
and, by arbitrariness of $P$, Mc Shane integrability of $F$ follows.
\end{proof}
\begin{remark} \rm \label{nota1}
If $F: T \rightarrow cwk(X)$ is Mc Shane integrable, then its integral coincides with
the following set:
\begin{eqnarray*}\label{int*}
\Phi(F,T) &=& \{ x \in X: \forall \, \varepsilon > 0, \exists \mbox{~a gauge~} \Delta :
\mbox{~~for every generalized \Pd} \\ & &
 \mbox{ Mc Shane partition}
(E_i,t_i)_{i\in \mathbb{N}} \mbox{~there holds:~}
\\ & &  \limsup_n \, d(x, \sum_{i=1}^{n} F(t_i) \mu(E_i)) \leq \varepsilon \}.
\end{eqnarray*}
(See \cite[Proposition 1]{bs2004}). The set $\Phi(F,T)$ is called the
 $(\star)$-integral of $F$.
No measurability is required a priori;
 moreover, if $F$ is single-valued, then $\Phi(F, T)$ coincides with the classical Mc Shane integral, if it exists.
\end{remark}

The simple Birkhoff integrability implies the Mc Shane one under suitable conditions: a first formulation holds in the countably additive case, for bounded functions $F$; a second theorem requires just finite additivity, but $F$ must be bounded and with compact support.
The pointwise non atomicity of  $\mu$ ensures that 
 all the partitions in the Henstock sense will turn out to be equivalent to the same concepts in the Mc Shane sense
i.e. without requiring that the  tags  $t_i$ are contained in the corresponding sets of the involved partitions
 (see also for example \cite[Proposition 2.3]{bcs2014}).

\begin{theorem}\label{wbvms-ca}
Suppose that $\mu$ is a pointwise non atomic  $\sigma$-finite measure.
 Let $F: T \rightarrow b(X)$ be a bounded  simple Birkhoff integrable multifunction.  
Then $F$ is also  Mc Shane integrable and the two integrals agree.
\end{theorem}
\begin{proof}
Thanks to the Lemma \ref{nota2}, it will be sufficient to prove that $F$ is u.t. Mc Shane integrable.
Let $M$ be an upper bound for $|F|$.
By simple Birkhoff integrability of $F$    there exists $I_w \in b(X)$ 
such that
for every $\varepsilon >0$, there exist a countable partition $P_{\varepsilon}:=\{E_{n}\}_{n\in \mathbb{N}}$
of $T$  so that for every countable
partition $P=\{A_{n}\}_{n\in \mathbb{N}}$ of $T$, with $P\geq P_{\varepsilon}$ and every $t_{n}\in A_{n},n\in \mathbb{N}$, one has 
\begin{eqnarray}\label{wb1}
 \limsup_n h\left(\sum_{i=1}^nF(t_i)\mu(A_i),I_w\right)\leq \vp/3.
\end{eqnarray}
Since $\mu$ is outer regular consider the sequence of open sets  $(G_n)_n$ associated to  $P_{\varepsilon}$ 
such that for every $n \in \mathbb{N}$, $\mu(G_n)<+\infty, 
E_n \subset G_n \in \mathscr{T} $ and $\mu(G_n \setminus E_n) \leq \varepsilon (M 4^n)^{-1}$.\\
 Let now $\Delta^* : T \rightarrow \mathscr{T}$ be defined by:
\[ \Delta^* (t) := \sum_{n \in \mathbb{N}} G_n 1_{E_n}.\]
Let now $\Pi \in {\cal P}_{\Delta^*}, \Pi := \{(B_n, w_n), n \in \mathbb{N}\}$ be any u.t. $\Delta^*$-fine Mc Shane partition. 
It will be proven that
$$\limsup_nh \left(\sum_{i=1}^n F(w_i)\mu(B_i),I_w\right) \leq \vp.$$
Since $\mu$ is non-atomic, without loss of generality $\Pi$ can be assumed to be of the Henstock type: otherwise it will suffice to replace every set $B_n$ with $(B_n\cup\{w_n\})\setminus C$, where $C$ is the set of all tags $w_n$.
For every $n \in \mathbb{N}$ let $j(n)$ be the only index such that
 $w_n \in E_{j(n)}$, and so $B_n \subset G_{j(n)}$. 
Suitably enumerating the sets $E_k$ , one can assume that the mapping $n\mapsto j(n)$ is increasing.
Split each $B_n$ into two parts:
$B_n \cap E_{j(n)}$ and $B_n \setminus E_{j(n)}$ in such a way that:
\begin{eqnarray}\label{2pezzi}
\nonumber
\mu(B_n) &=& \mu(B_n \cap E_{j(n)}) + \mu(B_n \setminus E_{j(n)}) \leq 
\mu(B_n \cap E_{j(n)}) + \mu(G_{j(n)}  \setminus E_{j(n)})\\ &\leq&
 \mu(B_n \cap E_{j(n)}) +
\varepsilon (M 4^{j(n)})^{-1}
\end{eqnarray}
Moreover, thanks to countable additivity of $\mu$, one has
$$\mu\left(\bigcup_n(B_n\setminus E_{j(n)})\right)\leq \vp M^{-1}\sum_{j=1}^{+\infty}4^{-j}=\vp(3M)^{-1}.$$
For further convenience, the set $\bigcup_n(B_n\setminus E_{j(n)})$ will be denoted as $Q$. Since $\mu(Q)\leq \vp(3M)^{-1}$, it is easy to see that 
\begin{eqnarray}\label{Bridotta}
h\left(\sum_{i=1}^n F(w_i)\mu(B_i),\sum_{i=1}^n F(w_i)\mu(B_i\cap E_{j(i)})\right)\leq M\mu(Q)\leq \frac{\vp}{3}, \end{eqnarray} for all integers $n$.
Now, to  every set $E_j$ of $P_{\vp}$ associate a tag $t_j$ as follows: $t_j=w_{j(n)}$ if $E_j=E_{j(n)}$ for some $n$, otherwise let $t_j$ be any arbitrary point of $E_j$. In this way $P_{\vp}$ becomes a Henstock-type partition, and certainly satisfies
\begin{eqnarray}\label{wb2}
 \limsup_n h\left(\sum_{j=1}^nF(t_j)\mu(E_j),I_w\right)\leq \vp/3.
\end{eqnarray}
So, there exists an integer $N'$ large enough that
\begin{eqnarray}\label{birkhoff}
h\left(\sum_{r=1}^{N''}F(t_r)\mu(E_r),I_w\right)\leq \vp/3
\end{eqnarray}
holds, for every integer $N''\geq j(N')$. In particular, if $N$ is any integer larger that $N'$, then $j(N)\geq j(N')$.

Except for a null set, $\Omega=Q\cup[\bigcup_n(B_n\cap E_{j(n)})]$,
hence, for every integer $k$, one has
$$E_k=(E_k\cap Q)\cup [E_k\cap B_1\cap E_{j(1)}]\cup [E_k\cap B_2\cap E_{j(2)}]\cup...$$
and, since the sets $E_j$ are pairwise disjoint, all the intersections $E_k\cap B_i\cap E_{j(i)}$ are null sets, unless $E_k=E_{j(i)}$ for one $i$ (and only one). So, if $E_k$ is not of the type $E_{j(i)}$, then (except for a null set), $E_k=E_k\cap Q$; while, otherwise, $E_k=(E_k\cap Q)\cup (B_i\cap E_{j(i)})$,
(where $k=j(i)$). 
So, chosen arbitrarily $N\geq N'$, this leads to
$$\sum_{r=1}^{j(N)}F(t_r)\mu(E_r)=\sum_{r=1}^{j(N)}
F(t_r)\mu(E_r\cap Q)+\sum_{i=1}^NF(t_i)\mu(B_i\cap E_{j(i)}).$$
This implies that
\begin{eqnarray}\label{passaggio}
h\left(\sum_{i=1}^NF(t_i)\mu(B_i\cap E_{j(i)}),\sum_{r=1}^{j(N)}F(t_r)\mu(E_r)\right)\leq M\mu(Q)\leq \frac{\vp}{3}.
\end{eqnarray}
So, from (\ref{Bridotta}), (\ref{passaggio}) and (\ref{birkhoff}), one deduces that
\begin{eqnarray*}
h\left(\sum_{i=1}^NF(t_i)\mu(B_i),I_w\right) &\leq&
h\left(\sum_{i=1}^N F(w_i)\mu(B_i),\sum_{i=1}^N F(w_i)\mu(B_i\cap E_{j(i)})\right)+
\\ &+&
h\left(\sum_{i=1}^NF(t_i)\mu(B_i\cap E_{j(i)}),\sum_{r=1}^{j(N)}F(t_r)\mu(E_r)\right)+
\\ &+&
h\left(\sum_{r=1}^{j(N)}F(t_r)\mu(E_r),I_w\right)\leq \vp.
\end{eqnarray*}
Since $N$ is arbitrarily chosen larger than $N'$, one gets
$$\limsup_nh\left(\sum_{i=1}^nF(t_i)\mu(B_i),I_w\right)\leq \vp.$$
and therefore u.t. Mc Shane integrability for $F$, by arbitrariness of the $\Delta^*$-fine partition $\Pi$. 
\end{proof}

So, it has been proved that, for $b(X)$-valued mappings, the simple-Birkhoff integrability implies the Mc Shane integrability, at least for bounded functions, likewise for the $cwk(X)$-valued case Birkhoff integrability implies Mc Shane integrability. Hence, simple-Birkhoff integrability is a natural extension of the Birkhoff one to the case of $b(X)$-valued functions.\\

The next Theorem \ref{wbvms-fa} does not involve directly countable additivity.
It is advisable to simplify notations and proofs when $\mu$ is a finite (countably additive) measure. Indeed, the following lemma holds true.
  In order to simplify notations, the symbol $s(P,F)$ will be adopted, to denote the summation $\sum_i F(t_i)\mu(A_i)$, whenever $P:=\{(A_i,t_i): i\in \mathbb{N}\} $ is a tagged partition.\\

\begin{lemma}\label{potyrala}
Assume that $(T,\mathcal{A},\mu)$ is a finite measure space, and $P:=\{(E_i,t_i):i\in \mathbb{N}\}$ is any Mc Shane partition,
and assume that $P$ is $\Delta$-fine, for some gauge $\Delta $ on $T$. Then, for every bounded function $F:T\to cwk(X)$ and every $\vp>0$ there exists a $\Delta$-fine finite partition $P':=\{(E'_j,t'_j)\}$ of $T$, in such a way that 
$$\limsup_nh\left(\sum_{j=1}^nF(t_i)\mu(E_i),\sum_{j=1}^NF(t'_i)\mu(E'_i)\right)\leq 2\vp.$$
\end{lemma}
\begin{proof}
Let $M$ denote any positive domination for $|F|$, and choose any integer $N$ such that $\sum_{i=N+1}^{+\infty}\mu(E_i)\leq \vp M^{-1}.$ Then define $P':=\{(E_j,t_j), j=1...,N, (B_j,w_j), \\ j=N+1,...,k\}$ in such a way that $P'':=\{(B_j,w_j), j=N+1,...,k\}$ is a $\Delta$-fine  Mc Shane partition of the set $\bigcup_{j=N+1}^{+\infty}E_j$ (its existence is due to the well-known Cousin Lemma,  see \cite[ Proposition 1.7]{riecan}). Then, for every $N'\geq N$ one has
\begin{eqnarray*}
h\left(\sum_{i=1}^{N'}F(t_i)\mu(E_i),s(P',F)\right) 
&\leq& h\left(\sum_{i=N+1}^{N'}F(t_i)\mu(E_i),s(P'',F)\right)\leq\\
&\leq&  h\left(\sum_{i=N+1}^{N'}F(t_i)\mu(E_i),\{0\}\right)+
h(s(P'',F),\{0\})\leq  2\vp.
\end{eqnarray*}
Since $N'\geq N$ is arbitrary, the conclusion is obvious.
\end{proof}

In some particular situations, the implication stated in the Theorem
\ref{wbvms-ca} holds also if the measure $\mu$ is just finitely additive, and with a much simpler proof.

\begin{theorem}\label{wbvms-fa}
Suppose that $\mu$ is a pointwise non atomic finitely additive regular
 measure. Let $F: T \rightarrow cwk(X)$ be a bounded  simple Birkhoff integrable multifunction  
 in $cwk(X)$ with compact support. Then $F$ is also  Mc Shane integrable and the two integrals agree.
\end{theorem}
\begin{proof}
Denote by $S$ the compact support of $F$. Of course, it is possible to replace $T$ with $S$ in the rest of the proof. Then $\mu$ is bounded in the $\sigma$-algebra $\mathcal{A}\cap S$, and also countably-additive there, as it is well-known. Then, by Theorem
 \ref{ex4.1}, $F$ is Gould-integrable with the same integral. Let $I$ denote the integral of $F$, and fix arbitrarily $\vp>0$. Then one can find  a finite partition $P_{\vp}:=\{E_i:i=1,...,n\}$ such that 
$$h(s(F,P'),I)=h\left(\sum_{j=1}^kF(t_j)\mu(A_j)\right)\leq \vp$$
holds true, for every finite partition $P':=\{(A_j,w_j):j=1,...,k\}$ finer than $P$ and every choice of points $w_j\in A_j$, $j\in\{1,...,k\}$.
Thanks to the regularity of $\mu$, there exist open sets $G_i$, $i=1,...,n$, such that $E_i\subset G_i$ and $\mu(G_i\setminus E_i)\leq \vp (nM)^{-1}$ for all indexes $i$.  Then define a gauge $\Delta$ on $S$, by setting, for all $t\in S$:
$$\Delta(t)=\sum_iG_i\ 1_{E_i}(t).$$
Now, choose any $\Delta$-fine Mc Shane finite partition $\Pi:=\{(B_r,v_r):r=1,...,l\}$.
Without loss of generality, one can assume that $\Pi$ is u.t. (otherwise simply glue together the sets $B_r$ that are associated to the same tag), and, by nonatomicity, that $v_r\in B_r$ for all $r$. Next, for all $r$ the set $G(v_r)$ contains $E_{i(r)}$ for a suitable index $i(r)$ and $\mu(G(v_r)\setminus E_{i(r)})\leq \vp (nM)^{-1}$, hence $\mu(B_r\setminus E_{i(r)})\leq \vp (nM)^{-1}$.
Take now the partition $\Pi'$ consisting of all the sets $B_r\cap E_{i(r)}$ and all the intersections $(B_r\setminus E_{i(r)})\cap E_l$, with $l\neq i(r)$. For all sets in $\Pi'$ choose tags as follows: to the sets of the type $B_r\cap E_{i(r)}$ associate the tag  $v_r$, while, for all the other sets, the tag is any arbitrary point  of the set.
Since $\Pi'$ is finer than $P_{\vp}$, one has
$h(s(\Pi',F),I)\leq \vp;$
but also
$$h(s(\Pi,F),s(\Pi',F))\leq M\sum_{r}\mu(B_r\setminus E_{i(r)})+M\sum_{r=1}^n\sum_{l\neq i(r)}\mu((B_r\setminus E_{i(r)})\cap E_l)\leq 2\vp$$
hence 
$h(s(\Pi,F),I)\leq 3\vp.$
Since $\Pi$ is arbitrary, this concludes the proof. 
\end{proof}

As to the converse implication, namely if Mc Shane integrability implies the Birkhoff simple one, it does not hold in general, as the next example shows:
\begin{example}\label{rodriguez}\rm
In \cite[Example 2.1]{r2009b} an example found by Phillips has been recalled, in order to show a mapping $f:[0,1]\to l^{\infty}([0,1])$ which is bounded, Mc Shane integrable and not Birkhoff-integrable. Thanks to Lemma \ref{potyrala} (and subsequent sentence), $f$ is not even simple Birkhoff integrable.
Clearly, this function can be used to define a multivalued mapping $F$, simply by setting $F(t)=\{f(t)\}$ for all $t\in [0,1]$. Of course, integrability (in any sense) of $F$ is equivalent to integrability (in the corresponding sense) of $f$. Thus the mapping $F$ gives an example of a bounded multivalued function, defined in a finite countably additive non-atomic measure space, which is Mc Shane integrable but not simple Birkhoff integrable. 
\end{example}

A partial result could be the following:
\begin{remark} \rm 
Let $(T,\mathscr{T}, {\cal A}, \mu)$ be a 
$\sigma$-finite quasi Radon outer regular measure
space and let $X$ a Banach space such that $B_{X^*}$ is weak* separable. 
Then a bounded multifunction \mbox{$F: T \rightarrow cwk(X)$}  with compact support is simple Birkhoff integrable if and only if it is Mc Shane integrable and the two integrals agree.
This is an easy consequence of Proposition \ref{wbvms-ca} and of the remark of \cite[Corollary 4.2]{bs2012}.
\end{remark}

\section*{Acknowledgements}
Two of the authors have been supported by University of Perugia -- Department of Mathematics and Computer Sciences -- Grant Nr 2010.011.0403, 
Prin: Metodi logici per il trattamento dell'informazione,   Prin: Descartes and by the Grant prot. U2014/000237 of GNAMPA - INDAM (Italy).


\end{document}